\documentclass[11pt]{amsart}
\usepackage{bm}
\usepackage{graphicx}
\usepackage{mathpazo}
\usepackage{latexsym}   
\usepackage{amssymb}    
\usepackage{amsmath}    
\usepackage{amsthm}     
\usepackage{hyperref}
\usepackage[margin=1in]{geometry}
\usepackage[cmtip,all,matrix,arrow,tips,curve]{xy}
\usepackage[lite]{amsrefs}
\usepackage{color}

\def\res{\bm{{\rm R}}}

\newcommand{\Q}{{\mathbb Q}}

\newcommand{\Z}{{\mathbb Z}}

\newcommand{\F}{{\mathbb F}}

\newcommand{\cO}{\mathcal{O}}

\DeclareMathOperator{\GL}{GL}
\newcommand{\fg}{{\mathfrak g}}
\newcommand{\ft}{{\mathfrak t}}

\newcommand{\A}{{\mathbb A}}
\newcommand{\Gm}{{\mathbb{G}}_m}
\newcommand{\bG}{\mathbf {G}}

\newcommand{\can}{\mathrm{can}}
\newcommand{\spl}{\mathrm{spl}}
\newcommand{\cT}{{\mathcal T}}
\newcommand{\rep}{\xi}

\newcommand\cX{{\mathcal X}}

\DeclareMathOperator{\Gal}{Gal}

\def\aff{{\mathbb A}}
\def\ff{{\mathbb F}}
\def\rat{{\mathbb Q}}
\def\gp{{\mathbb G}}
\def\bs{\backslash}
\def\one{\bm{1}}
\def\inv{^{-1}}
\def\dual{^\vee}
\def\twiddle{\sim}
\def\tensor{\otimes}
\def\iso{\cong}
\def\cross{\times}
\def\inject{\hookrightarrow}
\newcommand{\oneover}[1]{\frac{1}{#1}}
\newcommand{\half}[1]{\frac{#1}{2}}
\newcommand{\abs}[1]{{\left|#1\right|}}
\newcommand{\st}[1]{\left\{#1\right\}}
\def\integ{{\mathbb Z}}
\def\ra{\rightarrow}
\def\units{^\times}
\DeclareMathOperator{\frob}{Fr}
\DeclareMathOperator{\vol}{vol}
\DeclareMathOperator{\mat}{Mat}
\DeclareMathOperator{\gl}{GL}
\DeclareMathOperator{\SL}{SL}

\DeclareMathOperator{\tr}{tr}
\DeclareMathOperator{\gal}{Gal}
\DeclareMathOperator{\End}{End}
\def\wnum{\widetilde{\#}}

\def\der{\textrm{der}}
\def\serre{\textrm{SO}}	
\def\geom{\textrm{geom}}
\def\can{\textrm{can}}
\def\cris{\textrm{cris}}
\def\m{\textrm{M}}
\DeclareMathOperator{\G}{G}

\newtheorem{theorem}{Theorem}[section]
\newtheorem{lemma}[theorem]{Lemma}
\newtheorem{proposition}[theorem]{Proposition}
\newtheorem{corollary}[theorem]{Corollary}

\newtheorem{defn}[theorem]{Definition}

\newtheorem*{thm*}{Theorem}

\numberwithin{equation}{section}

\newcommand{\fc}{{\mathfrak c}}
\newcommand{\rk}{\operatorname{rank}}
\newcommand{\Aut}{\operatorname{Aut}}
\newcommand{\fp}{\mathfrak p}
\newcommand{\M}{\operatorname{M}}

\begin{document}

\title{Elliptic curves, random matrices and orbital integrals}

\author{Jeffrey D. Achter}
\address{Colorado State University, Fort Collins, CO}
\email{achter@math.colostate.edu}

\author{Julia Gordon}
\address{University of British Columbia, Vancouver, BC}
\email{gor@math.ubc.ca}

\thanks{JDA's research was partially supported by  grants from the
Simons Foundation (204164) and the NSA (H98230-14-1-0161 and H98230-15-1-0247). JG's research was supported by NSERC}

\begin{abstract}
An isogeny class of elliptic curves over a finite field is determined by a quadratic Weil polynomial. Gekeler has given a  product formula, in terms of congruence considerations involving that polynomial, for the size of such an isogeny class.  In this paper, we give a new, transparent proof of this formula; it turns out that this product actually computes an adelic orbital integral which visibly counts the desired cardinality. This answers a question posed by N. Katz in \cite{katz_lt}*{Remark 8.7}. 
\end{abstract}

\maketitle

\section{Introduction}\label{secintro}

The isogeny class of an elliptic curve over a finite field $\ff_p$ of
$p$ elements  is
determined by its trace of Frobenius; calculating the size of
such an isogeny class is a classical problem.  Fix a number $a$ with
$\abs a \le 2 \sqrt p$, and let $I(a,p)$ be the
set of all elliptic curves over $\ff_p$ with trace of Frobenius $a$; further
suppose that $p\nmid a$, so that the isogeny class is ordinary.

Motivated by equidistribution considerations, Gekeler derives the following description  of the size of 
$I(a,\ff_p)$ (\cite{gekeler03}; see also \cite{katz_lt}).    For each rational
prime $\ell\not = p$, let
\begin{equation}
\label{eqgekelerterms}
\nu_\ell(a,p) = \lim_{n\ra\infty}\frac{\#\st{\gamma \in
  \gl_2(\integ/\ell^n): \tr(\gamma) \equiv a \bmod \ell^n,
  \det(\gamma)\equiv p \bmod \ell^n}}{\#\SL_2(\integ/\ell^n)/\ell^n}.
\end{equation}

For $\ell=p$, let 
\begin{equation}\label{eqgekeler-p}
\nu_p(a,p) = \lim_{n\ra\infty}\frac{\#\st{\gamma \in
  \M_2(\integ/p^n): \tr(\gamma) \equiv a \bmod p^n,
  \det(\gamma)\equiv p \bmod p^n}}{\#\SL_2(\integ/p^n)/p^n}.
\end{equation}

On average, the number of elements of $\gl_2(\integ/\ell^n)$ with a
given characteristic polynomial is 
$\#\gl_2(\integ/\ell^n)/(\#(\integ/\ell^n)\units \cdot \ell^n)$.
Thus, $\nu_\ell(a,p)$ measures the departure of the frequency of the
event $f_\gamma(T) = T^2-aT+p$ from the average.

It turns out that \cite[Thm.\ 5.5]{gekeler03}
\begin{equation}
\label{eqgekeler}
\wnum I(a,p) = \half 1 \sqrt p \nu_\infty(a,p) \prod_\ell \nu_\ell(a,p),
\end{equation}
where 
\[
\nu_\infty(a,p) = \frac 2\pi\sqrt{1-\frac{a^2}{4p}},
\]
$\wnum I(a,p)$ is a count weighted by automorphisms \eqref{eqdefwnum}, and we note that the term $H^*(a,p)$ of \cite{gekeler03} actually computes $2\wnum I(a,p)$ (see \cite[(2.10) and (2.13)]{gekeler03} and 
\cite[Theorem 8.5, p. 451]{katz_lt}).
 This equation is almost
miraculous.  An equidistribution assumption about Frobenius
elements, which is so strong that it can't possibly be true, leads one
to the correct conclusion.

In contrast to the heuristic, the proof of \eqref{eqgekeler} is
somewhat pedestrian.  
Let $\Delta_{a,p} = a^2-4p$, let $K_{a,p} =
\rat(\sqrt{\Delta_{a,p}})$, and let $\chi_{a,p}$ be the associated
quadratic character.  Classically, the size of the isogeny class
$I(a,\ff_p)$ is given by the Kronecker class number
$H(\Delta_{a,p})$. Direct calculation \cite{gekeler03} shows
that, at least for unramified primes $\ell$, 
\[
\nu_\ell(a,p) = \oneover{1-\frac{\chi_{a,p}(\ell)}{\ell}}
\]
 is the term at $\ell$ in the Euler product expansion of
 $L(1,\chi_{a,p})$.  More generally, a term-by-term comparison shows that the right-hand side of \eqref{eqgekeler} computes $H(\Delta_{a,p})$.

Even though \eqref{eqgekeler} is striking and unconditional, one might still want a pure
thought derivation of it.  (We are not alone in this desire; Katz calls
attention to this question in \cite{katz_lt}.)
Our goal in the present paper is to provide a conceptual explanation
of \eqref{eqgekeler}, and to extend it to the case of ordinary elliptic curves over an arbitrary finite field $\ff_q$.

In a companion work, we will use a similar method to give an analogous product formula for the size of an isogeny class of simple ordinary principally polarized abelian varieties over a finite field.

Our method relies on the description, due to
Langlands (for modular curves) and Kottwitz (in general), of the
points on a Shimura variety over a finite field.  A consequence of
their study is that  one can calculate the cardinality of an ordinary isogeny class of elliptic curves over $\ff_q$ using orbital integrals on the finite adelic points of $\gl_2$ (Proposition \ref{proplk}).
Our main observation is that
one can, without explicit calculation,  relate each local factor
$\nu_\ell(a,q)$ to an orbital integral 
\begin{equation}\label{eqlocal}
\int_{G_{\gamma_\ell}(\rat_\ell)\backslash\gl_2(\rat_\ell)} \one_{\gl_{2}(\integ_\ell)}(x\inv
\gamma_\ell x)\, dx, 
\end{equation}
where $\gamma_\ell$ is an element of $\gl_2(\Q_\ell)$ of trace $a$ and determinant $q$, $G_{\gamma_\ell}$ is its centralizer in $\gl_{2}(\rat_\ell)$, and $\one_{\gl_{2}(\integ_\ell)}$ is the characteristic function of the maximal compact subgroup $\gl_{2}(\integ_\ell)$.
Here the choice of the invariant measure $dx$ on the orbit is crucial.  On one hand, the measure that is naturally related to Gekeler's numbers is the so-called \emph{geometric measure} 
(cf. \cite{langlands-frenkel-ngo}), which we review in \S \ref{subsub:steinberg}. On the other hand, this measure is inconvenient for computing the global volume term that appears in the formula of Langlands and Kottwitz. 
The main technical difficulty is the comparison, which should be well-known but is hard to find in the literature, between the geometric measure and the so-called \emph{canonical measure}.

We start (\S \ref{secdef}) by establishing notation and reviewing the Langlands-Kottwitz formula.  We define the relevant, natural measures in \S \ref{seclocalcalc}, and study the comparison factor between them in \S \ref{sec:comparison}. 
Finally, in \S \ref{sec:global}, we complete the global calculation.   The appendix (\S \ref{appali}), by S. Ali Altu\u{g}, analyzes the comparison of measures from a slightly different perspective.

As we were finishing this paper, the authors of \cite{davidetal15} shared their preprint with us; we invite the reader to consult that work for a different approach to Gekeler's random matrix model.

\subsection*{Notation}
Throughout, $\ff_q$ is a finite field of characteristic $p$ and
cardinality $q=p^e$.   Let $\rat_q$ be the unique unramified extension
of $\rat_p$ of degree $e$, and let
$\integ_q \subset \rat_q$ be its ring of integers.  We use $\sigma$ to
denote both the canonical generator of $\gal(\ff_q/\ff_p)$ and its
lift to $\gal(\rat_q/\rat_p)$.  

Typically, $G$ will denote the algebraic group $\gl_2$.  While many of our results admit immediate generalization to other reductive groups, as a rule we resist this temptation unless the statement and its proof require no additional notation. 

Shortly, we will fix a regular semisimple element $\gamma_0 \in G(\rat) = \gl_2(\rat)$; its centralizer will variously be denoted $G_{\gamma_0}$ and $T$.

Conjugacy in an (abstract) group is denoted by $\twiddle$.

\subsection*{Acknowledgment} We have benefited from discussions with
Bill Casselman, Clifton Cunningham, David Roe, William Sawin, and
Sug-Woo Shin. We are particularly grateful to Luis Garcia for sharing
his insights. It is a great pleasure to thank these people.

\section{Preliminaries}\label{secdef}

Here we collect notation concerning isogeny classes (\ref{sub:defisog}) as well as basic information on Gekeler's ratios (\ref{sub:gekdef}) and the Langlands-Kottwitz formula (\ref{sub:lk}).

\subsection{Isogeny classes of elliptic curves}\label{sub:defisog}

If $E/\ff_q$ is an elliptic curve, then its characteristic polynomial of Frobenius has the form $f_{E/\ff_q}(T) = T^2-a_{E/\ff_q}T+q$, where $\abs {a_{E/\ff_q}} \le 2 \sqrt q$.  Moreover, $E_1$ and $E_2$ are $\ff_q$-isogenous if and only if $a_{E_1/\ff_q} = a_{E_2/\ff_q}$.  In particular, for a given integer $a$ with $\abs a \le 2 \sqrt q$, the set 
\[
I(a,q) = \st{ E/\ff_q: a_{E/\ff_q} = a}
\]is a single isogeny class of elliptic curves over $\ff_q$.   Its weighted cardinality is
\begin{equation}
\label{eqdefwnum}
\wnum I(a,q) := \sum_{E \in I(a,q)} \oneover{\#\Aut(E)}.
\end{equation}
A member of this isogeny class is ordinary if and only if $p\nmid a$; henceforth, we assume this is the case.

Fix an element $\gamma_0\in G(\rat)$ 
with characteristic polynomial $$f_0(T) = f_{a,q}(T) := T^2-aT+q.$$  Newton
polygon considerations show that exactly one root of $f_{a,q}(T)$ is a
$p$-adic unit, and in particular $f_{a,q}(T)$ has distinct roots.
Therefore, $\gamma_0$ is regular semisimple.  Moreover, any other
element of $G(\rat)$ with the same characteristic polynomial is
conjugate to $\gamma_0$.

Let $K = K_{a,q} = \rat[T]/f(T)$; it is a quadratic imaginary field.
If $E \in I(a,q)$, then its endomorphism algebra is $\End(E)\tensor
\rat \iso K$.  The centralizer $G_{\gamma_0}$ of $\gamma_0$ in $G$ is
the restriction of scalars torus $G_{\gamma_0} \iso \res_{K/\rat}
\gp_m$.

If $\alpha$ is an invariant of an isogeny class, we will variously
denote it as $\alpha(a,q)$, $\alpha(f_0)$, or $\alpha(\gamma_0)$,
depending on the desired emphasis.

\subsection{The Steinberg quotient}\label{sub:steinberg}
We review the general  definition of the Steinberg quotient.
Let $G$ be a split, reductive group of rank $r$, with simply connected derived group $G^\der$ and Lie algebra $\fg$; further assumpe that $G/G^\der \iso \gp_m$. (In the case of interest for this paper, $G = \gl_2$,  $r=2$, and $G^\der = \SL_2$.)

Let $T$ be a split maximal torus in $G$, $T^\der = T \cap G^\der$ (note that $T^\der$ is {\em not} the derived group of $T$), and let $W$ be the Weyl group of $G$ relative to $T$.  Let $A^\der=T^\der/W$ be the Steinberg quotient for the semisimple group $G^\der$. It is isomorphic to the affine space of dimension $r-1$. 

Let $A=A^\der\times \gp_m$ be the analogue of the Steinberg quotient for the reductive group $G$, cf  
\cite{langlands-frenkel-ngo}*{}.  
We think of $A$ as the space of ``characteristic polynomials''.  
There is a canonical map
\begin{equation}
\label{diagsteinberg}
\xymatrix{G \ar[r]^{\fc} & A}
\end{equation}

Since $G/G^\der \iso \gp_m$, we have
\[
A \iso \aff^{r-1}\cross \gp_m \subset \aff^r.
\]

\subsection{Gekeler numbers}\label{sub:gekdef} 

We resume our discussion of elliptic curves, and let $G = \gl_2$.  As
in \S \ref{sub:defisog}, fix data $(a,q)$ defining an ordinary isogeny
class over $\ff_q$.  Recall that, to each finite prime $\ell$, Gekeler
has assigned a local probability $\nu_\ell(a,q)$
\eqref{eqgekelerterms}-\eqref{eqgekeler-p}.  We give a geometric 
interpretation of this ratio, as follows.

Since $G$ is a group scheme over $\Z$, for any finite prime $\ell$, we
have a well-defined group $G(\Z_\ell)$, which is a (hyper-special)
maximal compact subgroup of $G(\Q_\ell)$, as well as the ``truncated''
groups $G(\Z_\ell/ \ell^n)$ for every integer $n\ge 0$.

Recall that, given the fixed data $(a,q)$, we have chosen an
element $\gamma_0 \in G(\rat)$.  Since the conjugacy class of a semisimple element of a general linear group is determined by its characteristic polynomial, $\gamma_0$ is well-defined up to conjugacy.

Let $\ell$ be any finite prime (we allow the possibility $\ell=p$);
using the inclusion $\rat \inject \rat_\ell$ we identify $\gamma_0$
with an element of $G(\rat_\ell)$. 
In fact, if $\ell \not = p$, then
$\gamma_0$ is a regular semisimple element of $G(\integ_\ell)$. 

For a fixed (notationally suppressed) positive integer $n$, the average value of $\#\fc\inv(a)$, as $a$ ranges over $A(\integ_\ell/\ell^n)$, is
\[
\#G(\integ_\ell/\ell^n)/\#A(\integ_\ell/\ell^n).
\]
Consequently, we set
\begin{align}
\label{eqdeffinitegekeler}
\nu_{\ell,n}(a,q) = \nu_{\ell,n}(\gamma_0) &= 
\frac{\#\st{ \gamma \in G(\integ_\ell/\ell^n): \gamma \twiddle (\gamma_0 \bmod \ell^n)}}{\#G(\integ_\ell/\ell^n)/\#A(\integ_\ell/\ell^n)},
\intertext{
and rewrite \eqref{eqgekeler} (and extend it to the case of $\ff_q$) as}
\label{eqnewgekeler}
\nu_\ell(a,q) &= \lim_{n\ra\infty} \nu_{\ell,n}(a,q).
\end{align}
Here, we have exploited the fact that two semisimple elements of $\gl_2$ are
conjugate if and only if their characteristic polynomials are the
same.  Note that the denominator of \eqref{eqnewgekeler} coincides with that of Gekeler's definition \cite{gekeler03}*{(3.7)}. 
Indeed, 
\begin{equation}\label{eq:denom}
\#G(\integ/\ell^n)/\#A(\integ/\ell^n) = \frac{\ell(\ell-1)(\ell^2-1)\ell^{4n-4}}{(\ell-1)\ell^{n-1}\ell^n} = (\ell^2-1)\ell^{2n-2}.
\end{equation}

For $\ell=p$, $\gamma_0$ lies in $\gl_2(\rat_p)\cap \mat_2(\integ_p)$.  We make the apparently ad hoc definition
\begin{equation}
\label{eqnewgekeler-p}
\nu_p(a,q) = \lim_{n\ra\infty} \frac{\#\st{ \gamma \in
    \mat_{2}(\integ_p/p^n): \gamma \twiddle (\gamma_0 \bmod
    p^n)}}{\#G(\integ_p/p^n)/\#A(\integ_p/p^n)},
\end{equation}
where we have briefly used $\twiddle$ to denote similarity of matrices under the action of $\GL_2(\integ_p/p^n)$.  In the case where $q=p$, this recovers Gekeler's definition \eqref{eqgekeler-p}.

Finally, we follow \cite[(3.3)]{gekeler03} and, inspired by the Sato-Tate measure, define an archimedean term
\begin{equation}
\label{eqgekeler-infinity}
\nu_\infty(a,q) = \frac 2\pi\sqrt{1-\frac{a^2}{4q}}.
\end{equation}

\subsection{The Langlands and Kottwitz approach}\label{sub:lk}

For Shimura varieties of PEL type, Kottwitz proved \cite{kottwitz92}
Langlands's conjectural expression of the zeta function of that
Shimura variety in terms of automorphic L-functions on the associated
group.  A key, albeit elementary, tool in this proof is the fact that
the isogeny class of a (structured) abelian variety can be expressed
in terms of an orbital integral.  The special case where the Shimura
variety in question is a modular curve, so that the abelian varieties
are simply elliptic curves, has enjoyed several detailed presentations
in the literature (e.g., \cite{clozel:bourbaki}, \cite{scholze:lk}
and, to a lesser extent, \cite{achtercunningham:orb}), and so we
content ourselves here with the relevant statement.

As in \S \ref{sub:defisog}, fix data $(a,q)$ which determines an isogeny class of
ordinary elliptic curves over $\ff_q$, and let $\gamma_0 \in G(\rat)$ be a suitable choice.  If $E \in I(a,q)$, then for each
$\ell \nmid q$ there is an isomorphism $H^1(E_{\bar\ff_q},\rat_\ell)
\iso \rat_\ell^{\oplus 2}$ which takes the Frobenius endomorphism of
$E$ to $\gamma_0$.

There is an additive operator $F$ on
$H^1_{\operatorname{cris}}(E,\rat_q)$. 
   It is
$\sigma$-linear, in the sense that if $a \in \rat_q$ and $x \in
H^1_{\operatorname{cris}}(E,\rat_q)$, then $F(ax) = a^\sigma F(x)$.
To $F$ corresponds some $\delta_0 \in G(\rat_q)$, well-defined up to
$\sigma$-conjugacy.  (Recall that $\delta$ and $\delta'$ are
$\sigma$-conjugate if there exists some $h \in G(\rat_q)$ such that
$h\inv \delta h^\sigma = \delta'$.)
The two elements are related by
$\operatorname{N}_{\rat_q/\rat_p}(\delta_0) \twiddle \gamma_0$.

Let $G_{\gamma_0}$ be the centralizer of $\gamma_0$ in $G$.  Let $G_{\delta_0\sigma}$ be the twisted centralizer of $\delta_0$ in $G_{\rat_q}$; it is an algebraic group over $\rat_p$.

Finally, let $\aff^p_f$ denote the prime-to-$p$ finite adeles, and let $\hat\integ^p_f\subset \aff^p_f$ be the subring of everywhere-integral elements.  With these notational preparations, we have

\begin{proposition}
\label{proplk}
The weighted cardinality of an ordinary isogeny class of elliptic curves is given by
\begin{multline}
\label{eqlk}
\wnum I(a,q) = \\\vol(G_{\gamma_0}(\rat)\bs G_{\gamma_0}(\aff_f)) \cdot \int_{G_{\gamma_0}(\aff^p_f)\bs G(\aff^p_f)} \one_{G(\hat\integ^p_f)}(g\inv \gamma_0 g)\, dg \cdot \int_{G_{\delta_0\sigma}(\rat_p)\bs G(\rat_q)} \one_{G(\integ_q)\left(\begin{smallmatrix}1&0\\0&p\end{smallmatrix}\right) G(\integ_q)}(h\inv \delta_0 h^\sigma)\, dh.
\end{multline}
\end{proposition}

Here, each group $G(\rat_\ell)$ has been given the Haar measure which
assigns volume one to $G(\integ_\ell)$ (this is the so-called \emph{canonical measure}, see \S\ref{subsub:gross}).  The choice of nonzero Haar
measure on the centralizer $G_\gamma(\rat_\ell)$ is irrelevant, as
long as the same choice is made for the global volume computation.  Similarly, in the second, twisted orbital integral, $G(\rat_q)$ is given the Haar measure which assigns volume one to $G(\integ_q)$.
Since we shall need to able to say something about the volume term later, we need to fix the measures on  
$G_{\gamma_0}(\Q_\ell)$ for every $\ell$.
We choose the canonical measures $\mu^\can$ on both $G$ and $G_{\gamma_0}$ at every place. These measures are
defined below in \S\ref{subsub:gross}.

The idea behind Proposition \ref{proplk} is straight-forward.  (We defer to \cite{clozel:bourbaki} for details.)  Fix $E \in I(a,q)$ and $H^1(E_{\bar\ff_q},\rat_\ell) \iso
\rat_\ell^{\oplus 2}$ as above.  This singles out an integral
structure $H^1(E_{\bar\ff_q},\integ_\ell) \subseteq \rat_\ell^{\oplus
  2}$.  If $E'$ is any other member of $I(a,q)$, then the prime-to-$p$
part of an $\ff_q$-rational isogeny induces $E \ra E'$ gives a new
integral structure $H^1(E'_{\bar\ff_q},\integ_\ell)$ on
$\rat_\ell^{\oplus 2}$.  Similarly, $p$-power isogenies give rise to new integral structures on the crystalline cohomology $H^1_\cris(E,\rat_q)$.  In this way, $I(a,q)$ is identified with $K\units \backslash Y^p \cross Y_p$, where $Y^p$ ranges among $\gamma_0$-stable lattices in $Y^1(E_{\bar\ff_q},\aff^p)$, and $Y_p$ ranges among lattices in $H^1_\cris(E,\rat_q)$ stable under $\delta_0$ and $p\delta_0\inv$.  It is now straight-forward to use an orbital integral to calculate the automorphism-weighted, or groupoid, cardinality of the quotient set $K\units \backslash Y^p\cross Y_p$ (e.g., \cite[\S 6]{hales:bourbaki}).

We remark that most expositions of Proposition
\ref{proplk} refer to a geometric context in which
$\one_{G(\hat\integ^p_f)}$ is replaced with the characteristic
function of an open compact subgroup which is sufficiently small that objects have trivial automorphism groups, so that
the corresponding Shimura variety is a smooth and quasiprojective fine moduli space.
However, this assumption is not necessary for the counting argument underlying \eqref{eqlk}; see, for instance, \cite[3(b)]{clozel:bourbaki}.

\section{Comparison of Gekeler numbers with orbital integrals}\label{seclocalcalc}
The calculation is based on the interplay between several $G$-invariant measures on the adjoint orbits in $G$. We start by carefully reviewing the definitions and the normalizations of all Haar measures involved. 

\subsection{Measures on groups and orbits}\label{sub:measures}
Let $\pi_n: \Z_\ell \to \Z_\ell/\ell^n$ be the truncation map. 
For any $\Z_\ell$-scheme $\cX$, we denote by $\pi_n^{\cX}$ the corresponding map 
$$\pi_n^\cX: \cX(\Z_\ell)\to \cX(\Z_{\ell}/\ell^n)$$ induced by $\pi_n$. 

Once and for all, fix the Haar measure on $\A^1(\Q_\ell)$ such that the volume of $\Z_\ell$ is $1$. 
We will denote this measure by $dx$. 
The key observation that will play a role in our calculations is that with this normalization, the fibres of the standard projection 
$\pi_n^{\A^d}:\A^d(\Z_\ell)\to \A^d(\Z/\ell^n \Z)$ have volume $\ell^{-nd}$.

We also observe that there are two approaches to normalizing a Haar 
measure on the set of $\Q_\ell$-points of an arbitrary algebraic group $G$: one can either fix a maximal compact subgroup and assign volume $1$ to it; or one can fix 
a volume form  $\omega_G$ on $G$ with coefficients in $\Z$, and thus get the measure 
$\abs{\omega_G}_\ell$  on each $\bG(\Q_\ell)$. 
Additionally, the measure that will be the most useful for us to study Gekeler-type ratios is the so-called Serre-Oesterl\'e measure (which turns out to be the measure that B. Gross calls \emph{canonical} in \cite{gross:motive}, but which is different from, though closely related to, the measure that we call canonical here).

\subsubsection{Serre-Oesterl\'e measure}\label{subsub:SO}
Let $\cX$ be a smooth scheme over $\Z_\ell$. 
Then there is the so-called Serre-Oesterl\'e measure on $X$, which we will denote by 
$\mu_X^\serre$. 
It is defined in \cite{serre:chebotarev}*{\S 3.3}, see also \cite{veys:measure} for an attractive equivalent definition. For a smooth scheme that has a non-vanishing gauge form this definition coincides with the definition of A. Weil \cite{weil:adeles}, and 
by \cite{weil:adeles}*{Theorem 2.2.5} 
(extended by Batyrev \cite{batyrev:calabi-yau}*{Theorem 2.7}),
this measure has the property that 
$\vol^\serre(\cX(\Z_\ell))=\#\cX(\F_\ell)\ell^{-d}$, where $d$ is the dimension of the generic fiber of $\cX$.
In particular,
$\mu^\serre_{\A^1}$ is the Haar measure on the affine line such that 
$\vol^\serre(\A^1(\Z_\ell))=\ell \ell^{-1}=1$, i.e., $\mu^\serre_{\A^1(\Q_\ell)}$ coincides with 
$\abs{dx}_\ell$. Similarly, on any $d$-dimensional affine space 
$\A^d$, the Serre-Oesterl\'e measure gives $\A^d(\Z_\ell)$ volume $1$. 

The algebraic group $\gl_2$ is a smooth group scheme defined over $\Z$; in particular, for every 
$\ell$, 
$\gl_2\times_{\Z}\Z_\ell$ is a smooth scheme over $\Z_\ell$, so  
$\mu^\serre$ gives $\gl_2(\Z_\ell)$ volume 
$$\mu^\serre(\gl_2(\Z_\ell))=\frac{\# \gl_2(\F_\ell)}{\ell^d} =\frac{\ell(\ell-1)(\ell^2-1)}{\ell^4}.$$

\subsubsection{The canonical measures}\label{subsub:gross}
Let $G$ be a reductive group over $\Q_\ell$; then Gross \cite[Sec.\ 4]{gross:motive}
constructs a canonical volume form $\omega_G$, which does not vanish
on the special fibre of the canonical model $\underline
G$ over $\integ_\ell$.  If $G$ is unramified over $\rat_\ell$, then $\omega_G$
recovers the Serre-Oesterl\'e measure, insofar as 
\[
\int_{\underline{G}(\integ_\ell)} \abs{\omega_G}_\ell =
\frac{\#\underline{G}^\kappa(\ff_\ell)}{\ell^{\dim G}}, 
\]
where $\underline{G}^\kappa$ is the closed fibre of $\underline G$ \cite[Prop.\ 4.7]{gross:motive}.

The measure most commonly used in the calculation of orbital integrals, traditionally denoted by 
$\mu^\can$, is closely related to $\abs{\omega_G}$.
By definition, $\mu^\can$ is normalized by giving volume $1$ to the canonical compact subgroup 
$\underline G(\Z_\ell)$. 
For an unramified group $G$, 
this is a hyperspecial maximal compact subgroup $\underline G(\Z_\ell)$. 
Thus, for $G=\gl_2$, 
we have:
$\mu^\can(G(\Z_\ell))=1$. 

We will also need to understand the canonical measures for tori in $G$, which arise as the centralizers of semisimple elements. 
For a torus $T$, by definition, 
$\mu_T^\can(\underline{T}^\circ(\Z_\ell))=1$, where
$\underline{T}^\circ$ is the connected component of the canonical
integral model for $T$ (which will be discussed in more detail in \S \ref{sub:canonical}
below).

\subsubsection{The geometric measure}\label{subsub:steinberg} 
We will use a certain quotient measure $\mu^\geom$ on the orbits, which is called the geometric measure in \cite {langlands-frenkel-ngo}.
This measure is defined using the Steinberg map $\fc$ \eqref{diagsteinberg}; we return to the setting of \S \ref{sub:steinberg}.

For a general reductive group $G$ and $\gamma\in G(\Q_\ell)$ regular semisimple, the fibre over $\fc(\gamma)$ 
is the stable orbit of $\gamma$, which is a finite union of rational orbits. 
In our setting with $G=\GL_2$, the fibre $\fc^{-1}(\fc(\gamma))$ is a single rational orbit, which substantially simplifies the situation. From here onwards, we work only with $G=\GL_2$.  

Consider the measure given by the form $\omega_G$ on $G$, and 
the measure on $A = \A^{1}\times \Gm$ which is the product of the measures associated with the form 
$dt$ on $\A^{1}$ and $ds/s$ on $\Gm$, where we  denote the coordinates on $A$ by $(t,s)$. 
We will denote this measure by $\abs{d\omega_A}$.
 
The form $\omega_G$ is a generator of the top exterior power of  the 
cotangent bundle of $G$.
For each orbit $\fc^{-1}(t,s)$ (note that such an orbit is a variety) there is a unique generator $\omega^{\geom}_{\fc(\gamma)}$ of the top exterior power of the cotangent bundle on the orbit $\fc^{-1}(\fc(\gamma))$ such that 
$$\omega_G=\omega^{\geom}_{\fc(\gamma)}\wedge \omega_{A}.$$
Then for any $\phi\in C_c^\infty(G(\Q_\ell))$,
$$
\int_{G(\Q_\ell)} \phi(g)\, d\abs{\omega_G}=\int_{A(\Q_\ell)}\int_{\fc^{-1}(\fc(\gamma))}\phi(g)\,d\abs{\omega^{\geom}_{\fc(\gamma)}} 
\,\abs{d\omega_A(t,s)}.$$
This measure also appears in \cite{langlands-frenkel-ngo}, and it is discussed in detail in \S \ref{sec:comparison} below. 

\subsubsection{Orbital integrals}\label{subsub:oi}
There are two kinds of orbital integrals that will be relevant for us; they differ only in the normalization of measures on the orbits. 
Let $\gamma$ be a regular semisimple element of $G(\Q_\ell)$, and let $\phi$ be a locally constant compactly supported function on $G(\Q_\ell)$. 
Let $T$ be the centralizer $G_\gamma$ of $\gamma$. 
Since $\gamma$ is regular (i.e., the roots of the characteristic polynomial of $\gamma$ are distinct) and semisimple, $T$ is a maximal torus in $G$.

First,  we consider the orbital integral with respect to the geometric measure:
\begin{defn} Define $O^{\geom}_{\gamma}(\phi)$ by 
\begin{equation*}
O^{\geom}_{\gamma}(\phi) :=\int_{T(\rat_\ell)\backslash G(\Q_\ell)} \phi(g^{-1}\gamma g) d\mu^{\geom}_{\gamma},
\end{equation*}
where $\mu^{\geom}_{\gamma}$ is the measure on the orbit of $\gamma$ associated with the corresponding differential form 
$\omega^{\geom}_{\fc^{-1}(\fc(\gamma))}$ 
as in  \S\ref{subsub:steinberg} above. 
\end{defn} 

Second, there is the canonical orbital integral over the orbit of $\gamma$, defined as follows. 
The orbit of $\gamma$ can be identified with the quotient
$T(\rat_\ell)\backslash G(\Q_\ell)$.  Both $T(\rat_\ell)$ and $G(\Q_\ell)$ are endowed with canonical measures, as above in \S\ref{subsub:gross}. Then there is unique quotient measure on $T(\rat_\ell)\backslash G(\Q_\ell)$, which
will be denoted $d\mu_\gamma^\can$. 
The canonical orbital integral will be the integral with respect to this measure on the orbit 
(also considered as  a distribution on the space of locally constant compactly supported functions on $G(\Q_\ell)$): 
\begin{defn} Define $O_{\gamma}^\can(\phi)$ by 
\begin{equation*}
O_{\gamma}^{\can}(\phi) :=\int_{T(\rat_\ell)\backslash G(\Q_\ell)} \phi(g^{-1}\gamma g) d\mu_\gamma^\can.
\end{equation*}
\end{defn}

By definition, the distributions $O^{\geom}_\gamma$ and $O_\gamma^\can$ differ by a multiple that is 
a function of $\gamma$. This ratio (which we feel should probably be well-known but was hard to find in the literature, see also \cite{langlands-frenkel-ngo}) is computed in \S \ref{sec:comparison} below.

We will first relate Gekeler's ratios to orbital integrals with respect to the geometric measure, in a natural way, and from there will get the relationship with the canonical orbital integrals, which are more convenient to use for the purposes of computing the global volume term appearing the formula of Langlands and Kottwitz.

\subsection{Gekeler numbers and volumes, for $\ell$ not equal to $p$}\label{sub:gek_p_not_l}
From now on, $G=\gl_2$, $\gamma_0 =\gamma_{a,q}$, and $\ell$ is a fixed prime distinct from $p$.
Our first goal is to relate the  Gekeler
number $\nu_\ell(a,q)$  \eqref{eqnewgekeler} 
to an orbital integral  $O^{\geom}_{\gamma_0}(\phi_0)$ of a suitable test function  $\phi_0$
 with respect to $d\abs{\omega^{\geom}_{\fc(\gamma)}}$. (Recall that $\gamma_0$ is the element of $G(\Q_\ell)$ determined by 
 $E$, and in this case since $\ell\neq p$, it lies in $G(\Z_\ell)$.)
 In order to do this we define natural subsets of $G(\Q_\ell)$ whose volumes are responsible for this relationship.
 
Recall \eqref{eqdeffinitegekeler} the definition of $\nu_{\ell,n}(\gamma_0)$.
For each positive integer $n$, consider the subset $V_n$ of $\GL_2(\Z_\ell)$ defined as 
\begin{align}\label{eq:subsets}
V_n= V_n(\gamma_0)&:=\st{\gamma\in \GL_2(\Z_\ell)\mid f_\gamma(T)\equiv f_0(T) \bmod  \ell^n} \\
&= \st{ \gamma \in \gl_2(\integ_\ell) \mid \pi_n^A(\fc(\gamma)) = \pi_n^A(\fc(\gamma_0))}.
\intertext{and set}
V(\gamma_0) &:= \cap_{ n \ge 1} V_n(\gamma_0).
\end{align}

We define an auxiliary ratio:
\begin{equation}\label{eq:vn}
v_n(\gamma_0):= \frac{\vol_{\mu_{GL_2}^\serre}(V_n(\gamma_0))}{\ell^{-2n}}.
\end{equation}

Now we would like to  relate the limit of these 
ratios $v_n(\gamma_0)$ both to the limit of Gekeler  ratios $\nu_{\ell,n}(\gamma_0)$ 
and to an orbital integral.

Let $\phi_0 = \one_{\gl_2(\integ_\ell)}$ be the characteristic function of the maximal compact subgroup
$\GL_2(\Z_\ell)$ in $\GL_2(\Q_\ell)$. 

\begin{proposition}\label{prop:st.oi}
We have
\[
\lim_{n\to \infty} v_n(\gamma_0)= 
O_{\gamma_0}^{\geom}(\phi_0).
\]
\end{proposition}

\begin{proof}
Because equality of characteristic polynomials is equivalent to conjugacy in $\gl_2(\rat_\ell)$, $V(\gamma_0)$ is the intersection of $\gl_2(\integ_\ell)$ with the orbit $\cO(\gamma_0)$ of $\gamma_0$ in $G=\GL_2(\Q_\ell)$. 
Then the orbital integral
$O_{\gamma_0}^{\geom}(\phi_0)$
is nothing but the volume of the set $V(\gamma_0)$, as a subset of $\cO(\gamma_0)$, with respect to the 
measure $d\mu^{\geom}_{\gamma_0}$.

Let $a_0=\fc(\gamma_0)= (a,q)\in \A^1\times \Gm(\Q_\ell)$, and let $U_n(a_0)$ be its $\ell^{-n}\times \ell^{-n}$-neighborhood.  Its Serre-Oesterle volume is $\vol_{\mu^\serre_A}(U_n(\gamma_0)) = \ell^{-2n}$.

Moreover, $V_n(\gamma_0) = \fc\inv(U_n(\gamma_0)) \cap \gl_2(\integ_\ell)$.  Consequently,
\begin{equation}\label{eq:st.volume}
\begin{aligned}
\lim_{n\to \infty}v_n(\gamma_0)&=\lim_{n\to \infty}\frac{\vol_{\mu^\serre_{\GL_2}}(\fc^{-1}(U_n(\gamma_0)) \cap \GL_2(\Z_\ell))}{\vol_{\mu^\serre_A}(U_n(\gamma_0))} \\
&=\lim_{n\to \infty}\frac{\vol_{|d\omega_G|}(\fc^{-1}(U_n(\gamma_0)) \cap \GL_2(\Z_\ell))}{\vol_{|d\omega_A|}(U_n(\gamma_0))}
=\vol_{\mu^{\geom}_{\gamma_0}}(V(\gamma_0)),
\end{aligned}
\end{equation} 
by  definition of the geometric measure.  
\end{proof}

Next, let us relate the ratios $v_n$ to the Gekeler ratios.

\begin{proposition}\label{prop:gek.volume}
We have the relation
$$
\lim_{n\to \infty} v_n(\gamma_0)=\frac{\#\SL_2(\F_\ell)}{\ell^3}\lim_{n\to\infty}\nu_{\ell, n}(\gamma_0) = \frac{\ell^2-1}{\ell^2}\nu_\ell(a,q).
$$ 
\end{proposition}

\begin{proof}
Let $\pi_n = \pi_n^{\gl_2}:\GL_2(\Z_\ell)\to \GL_2(\Z/\ell^n)$.  To ease notation slightly, let $V_n = V_n(\gamma_0)$.  
Let $S_n\subset \GL_2(\Z/\ell^n \Z)$ be the set that appears in the numerator \eqref{eqdeffinitegekeler}:
$$S_n:= \st{\gamma\in \GL_2(\Z/\ell^n)\mid f_\gamma(T)\equiv f_0(T)\bmod \ell^n}.$$
We claim that, for large enough $n$ (depending on the discriminant of $f$), the following hold:
\begin{enumerate}
\def\theenumi{\roman{enumi}}
\item $V_n=\pi_n^{-1}(S_n)$ and $\pi_n\vert V_n :V_n \to S_n$ is surjective;
\item 
We have the equality
\begin{equation}\label{eq:key}
\vol_{\mu_{\GL_2}^\serre}(V_n)=\ell^{-4n}\#S_n. 
\end{equation}
\end{enumerate}
The first claim is needed to establish the second one; and provided the second  claim holds, we get
(where the denominator of Gekeler's ratio is handled as in (\ref{eq:denom}) above):
\begin{equation}\label{eq:main}
v_n(\gamma_0) =\frac{\ell^{-4n} \#S_n}{\ell^{-2n}}= \frac{\#S_n}{\ell^{2n}}
=\frac{\#S_n \#\SL_2(\F_\ell)}{\#\SL_2(\F_\ell)\ell^{3(n-1)}\ell^{-n}\ell^3} =   
\frac{\#\SL_2(\F_\ell)}{\ell^3}\nu_{\ell,n}(\gamma_0),
\end{equation}
as required. 
Thus, it remains to verify the two claims. 

{\bf (i).}  Taking characteristic polynomials commutes with reduction $\bmod \ell^n$, since the coefficients of the characteristic polynomial are themselves polynomial in the matrix entries of $\gamma$, and reduction $\mod \ell^n$ is a ring homomorphism. Thus, $\pi_n^{-1}(S_n)=V_n$. 
Since $\gl_2$ is
a smooth scheme over $\integ_\ell$, $\pi_n$ is surjective for all $n$.

{\bf (ii)} 
Since $\GL_2$ is smooth over the residue field $\F_\ell$, all fibres
of $\pi_n$ have volume equal to $\ell^{-4 n}$.
By part (i), the set $V_n$ is a disjoint union of fibres of $\pi_n$, and the number of these fibres is $\#\pi_n(V_n)=\#S_n$. 
Thus, the volume of $V_n$ is exactly $\ell^{-4n}$ times the number of points in the image of the set in the numerator under this projection. 
\end{proof} 

Combining Propositions \ref{prop:st.oi} and \ref{prop:gek.volume}, we
immediately obtain:

\begin{corollary}\label{cor:gekeler-steinberg} 
The Gekeler numbers relate to orbital integrals via 
$$\nu_\ell(a,q)=\frac{\ell^3}{\#\SL_2(\F_\ell)}O^{\geom}_{\gamma_0}(\phi_0).$$
\end{corollary}

\subsection{$\ell=p$ revisited}

We now consider $\nu_p(a,q)$ in a similar light.  Since
$\det(\gamma_0) = q$, $\gamma_0$ lies in $\mat_2(\integ_p) \cap
\gl_2(\rat_p)$ but  {\em not} in $\gl_2(\integ_p)$, and we must
consequently modify the argument of \S \ref{sub:gek_p_not_l}.

For integers $m$ and $n$, let $\lambda_{m,n} = \begin{pmatrix}
  p^m&0\\0 & p^n\end{pmatrix}$, and let $C_{m,n} = \gl_2(\integ_p)
\lambda_{m,n} \gl_2(\integ_p)$.  The Cartan decomposition for $\gl_2$
asserts that $\gl_2(\rat_p)$ is the disjoint union
\begin{align*}
\gl_2(\rat_p) &= \bigcup_{m\ge n} C_{m,n},
\intertext{so that}
\mat_2(\integ_p) \cap \gl_2(\rat_p) &= \bigcup_{0 \le n \le m} C_{m,n}.
\end{align*}

We now express $\nu_p(a,q)$ as an orbital integral.  Recall that $q =
p^e$.  Since we consider an ordinary isogeny class, 
the element $\gamma_0 \in \GL_2(\Q_p)$ actually can be chosen to have the form 
$\gamma_0=\begin{pmatrix}
  u_1p^e&0\\0 & u_2\end{pmatrix}$, where $u_1, u_2\in \Z_p$ are units, and thus
in particular,  $\gamma_0 \in C_{e, 0}$.

\begin{lemma}\label{lem:l=p,g=1} 
Let $\phi_q$ be the characteristic function of $C_{e, 0} =
\gl_2(\integ_p)\begin{pmatrix} q&0\\0&1\end{pmatrix}
\gl_2(\integ_p)$.  Then
$$\nu_p(a,q)= \frac{p^{3}}{\#\SL_2(\F_p)}O^{\geom}_{\gamma_0}(\phi_q).$$
\end{lemma}  
\begin{proof}
The proof is similar to the case $\ell\neq p$, with one key modification. There, we are using the 
reduction $\bmod \ell^n$ map $\pi_n$ defined on $G(\Z_{\ell})$. 
Here, we need to extend the map $\pi_n$ to a set that contains
$\gamma_0$. 

Let $\pi^\m_n: \mat_2(\integ_p) \to \mat_2(\integ_p/p^n)$ be the
projection map, and let $\fc: \gl_2(\rat_p) \ra A(\rat_p)$ be the
characteristic polynomial map.  As in \S \ref{sub:gek_p_not_l} above, we  define 
the sets 
\begin{align*}
U_n&:= \{a=(a_0, a_1)\in A(\Z_p) \mid a_i \equiv a_i(\gamma_0) \mod p^n , i=0,1\}\\
S_n&:=\{\gamma\in \mat_2(\Z_p/p^n): \gamma\twiddle \pi_n^\m(\gamma_0)\}\\
V_n&:= (\pi_n^\m)^{-1}(S_n)\subset \mat_2(\Z_p)\cap \GL_2(\Q_p). 
\end{align*}
As before, informally, we think of $U_n$ as a neighbourhood of the point given by the coefficients of the characteristic polynomial of $\gamma_0$ in the Steinberg-Hitchin base, and we think of $V_n$ as the intersection of the corresponding neighbourhood of the orbit of $\gamma_0$ in $\GL_2(\Q_p)$ with $\mat_2(\Z_p)$. 
In the case $\ell \neq p$  we had $\GL_2(\Z_\ell)$ in the place of $\mat_2(\Z_p)$ in this description, and so  it was clear that the evaluation of the volume of $V_n$ would lead to the orbital integral of $\phi_0$, the characteristic function of $\GL_2(\Z_\ell)$. Here, we need to make the connection between the set $V_n$ and our function $\phi_q$.   

We claim that if $n>e$, then
$V_n \subset C_{e,0}$. Indeed, suppose $\gamma \in V_n$. Then, since the characteristic polynomial of $\gamma$ is congruent to that of $\gamma_0$, the  trace of $\gamma$ is a $p$-adic unit. 
Then $\gamma$ cannot lie in any double coset $C_{m,n}$ with both $m, n$ positive, because if it did, its trace would have been divisible by $p^{\min (m,n)}$. 
Then $\gamma$ has to lie in a double coset of the form $C_{e+m, -m}$ for some $m\ge 0$; but if $m>0$, then such a double coset has empty intersection with $\mat_2(\Z_p)$, so $m=0$ and the claim is proved. 

As in the proof of Proposition \ref{prop:gek.volume} (iii), the volume
of the set $V_n$ equals $p^{-4n}\# S_n$.  The rest of the proof
repeats the proofs of Proposition \ref{prop:gek.volume} and Corollary
\ref{cor:gekeler-steinberg}.  We again set $V(\gamma_0)=\cap_{n\ge 1}
V_n\subset C_{e,0}$.   Since $\pi^\m_n$ is surjective, $V(\gamma_0)=O(\gamma_0)\cap C_{e,0}$.  By (\ref{eq:st.volume}),
$$O^{\geom}_{\gamma_0}(\phi_q)=\lim_{n\to \infty}\frac{\vol_{\mu^\serre_{\GL_2}} V_n(\gamma_0)}{\vol_{\mu^\serre_A}(U_n)}=
\lim_{n\to\infty}\frac{\#S_n(\gamma_0)p^{-4n}}{p^{-2n}},$$ and the statement follows 
by (\ref{eq:main}), which does not require any modification. 

\end{proof}

Recall that, in terms of the data $(a,q)$, we have also computed a representative $\delta_0$ for a $\sigma$-conjugacy class in $\gl_2(\rat_q)$.  It is characterized by the fact that, possibly after adjusting $\gamma_0$ in its conjugacy class, we have  $\operatorname{N}_{\rat_q/\rat_p}(\delta_0) = \gamma_0$.

The twisted centralizer $G_{\delta_0\sigma}$ of $\delta_0$ is an inner form of the centralizer $G_{\gamma_0}$ \cite[Lemma 5.8]{kottwitz82}; since $\gamma_0$ is regular semisimple, $G_{\gamma_0}$ is a torus, and thus $G_{\delta_0\sigma}$ is isomorphic to $G_{\gamma_0}$.  Using this, any choice of Haar measure on $G_{\delta_0\sigma}(\rat_p)$ induces one on $G_{\gamma_0}(\rat_p)$.

If $\phi$ is a function on $G(\rat_q)$, denote its twisted (canonical) orbital integral along the orbit of $\delta_0$ by
\[
TO_{\delta_0}^\can(\phi) = \int_{G_{\delta_0\sigma}(\rat_p)\bs G(\rat_q)} \phi(h\inv \delta_0 h^\sigma)\, d\mu^\can.
\]

\begin{lemma}
\label{lem:basechange}
Let $\phi_{p,q}$ be the characteristic function of $\gl_2(\integ_q) \lambda_{0,1} \gl_2(\integ_q)$.  Then
\[
TO_{\delta_0}^\can(\phi_{p,q}) = O^\can_{\gamma_0}(\phi_q).
\]
\end{lemma}

\begin{proof}
The asserted matching of twisted orbital integrals on $\gl_2(\rat_q)$ with orbital integrals on $\gl_2(\rat_p)$ is one of the earliest known instances of the fundamental lemma (\cite{langlands:basechange}; see also \cite[Sec.\ 4]{laumon:vol1}, \cite{getz-goresky}*{(E.4.9)} or even \cite[Sec.\ 2.1]{achtercunningham:orb}).
Indeed, the base change homomorphism of the Hecke algebras matches the characteristic function of 
$\gl_2(\integ_q) \lambda_{1,0} \gl_2(\integ_q)$ with $\phi_q+ \phi$, where $\phi$ is a linear combination of the characteristic functions of $C_{a,b}$ with $a+b=e$ and $a, b>0$. 
As shown in the proof of the previous lemma, the orbit of $\gamma_0$ does not intersect the double cosets $C_{a,b}$ with $a, b>0$, and thus the only non-zero term on the right-hand side is $O^\can_{\gamma_0}(\phi_q)$. 
\end{proof}

\section{Canonical measure vs. geometric measure}\label{sec:comparison}
Finally, we need to relate the orbital integral with respect to the geometric measure as above to the canonical orbital integrals. 
A very similar calculation is discussed in \cite{langlands-frenkel-ngo} (and as the authors point out, surprisingly, it seemed impossible to find in earlier literature).
Since our normalization of local measures seems to differ by an
interesting constant from that of \cite{langlands-frenkel-ngo} at
ramified finite primes, we carry out this calculation in our special
case.

\subsection{Canonical measure and $L$-functions}
\label{sub:canonical}
Here we briefly review the facts that go back to the work of Weil, Langlands, Ono, Gross, and many others, that 
show the relationship between convergence factors that can be used for Tamagawa measures and various Artin $L$-functions. Our goal is to introduce the Artin $L$-factors that naturally appear in the computation of the canonical measures. 
To any reductive group $G$ over $\Q_\ell$, Gross attaches a motive 
$M = M_G$ \cite{gross:motive}; following his notation, we consider
$M^\vee(1)$ -- the Tate twist of the dual of $M$. For any motive $M$ we let
$L_\ell(s,M)$ be the associated local Artin $L$-function. 
We will write $L_\ell(M)$ for the value of $L_\ell(s,M)$ at $s=0$.
The value $L_\ell(M^\vee(1))$ is always a positive rational number, related
to the canonical measure reviewed in \S \ref{subsub:gross}.  In
particular, if $G$ is quasi-split over $\rat_\ell$, then
\begin{equation}
\mu_G^\can=L_\ell(M^\vee(1))\abs{\omega_G}_\ell
\end{equation}
(\cite[4.7 and 5.1]{gross:motive}).

We shall also need a similar relation between volumes and Artin
$L$-functions in the case when $G=T$ is an algebraic torus which is
not necessarily anisotropic. Here we follow
\cite{bitan}. Suppose that $T$ splits over a finite Galois extension
$L$ of $\Q_\ell$; let $\kappa_L$ be the residue field of $L$, and let
$I$ 
be the inertia subgroup of the Galois group $\gal(L/\Q_\ell)$.  Let
$X^\ast(T)$ be the group of rational characters of $T$.  Let $\cT$ be
the N\'eron model of $T$ over $\Z_\ell$, with the connected component
of the identity denoted by $\cT^\circ$.
This is the canonical model for $T$ referred to in \ref{subsub:gross}.

Let $\frob_L$ be the Frobenius element of $\Gal(\kappa_L/\F_\ell)$.
The Galois group of the maximal unramified sub-extension of $L$, which
is isomorphic to $\Gal(\kappa_L/\F_\ell)$, acts naturally on the
$I$-invariants $X^\ast(T)^I$, giving rise to a 
representation
which we will denote by $\rep_T$ (and which is denoted by $h$ in
\cite{bitan}),
$$\rep_T:\Gal(\kappa_L/\F_\ell)\to \Aut(X^\ast(T)^I)\simeq \GL_{d_I}(\Z),$$ 
where $d_I=\rk(X^\ast(T)^I)$.  
Then the associated local Artin $L$-factor is defined as:
$$L_\ell(s, \rep_T):=\det\left(1_{d_I}-\frac{\rep_T(\frob_L)}{\ell^s}\right)^{-1}.$$
\begin{proposition}(\cite{bitan}*{Proposition 2.14})
$$L_\ell(1, \rep_T)^{-1} = \#\cT^\circ(\F_\ell)\ell^{-\dim(T)} = \int_{\cT^\circ(Z_\ell)}
\abs{\omega_{T}}_\ell.$$  
\end{proposition}

We observe that by definition  \cite{gross:motive}*{\S 4.3}, since $G=T$ is an algebraic torus, the canonical parahoric $\underline{T}^\circ$ is $\cT^\circ$; the canonical volume form 
$\omega_T$ is the same as the volume form denoted by $\omega_{\fp}$ in \cite{bitan}.

We also note that the motive of the torus $T$ is the Artin motive 
$M=X^\ast(T)\otimes \Q$. 
If $T$ is anisotropic over $\Q_\ell$, 
by the formula  (6.6) (cf. also (6.11)) in \cite{gross:motive}, 
we have 
$$L_\ell(M^\vee(1)) = L_\ell(1, \rep_T).$$

As in the first paragraph of \S \ref{subsub:steinberg}, let
$G$  be a reductive group over $\rat_\ell$ with simply connected derived group
$G^\der$ and connected center $Z$, and assume that $G/G^\der \iso \gp_m$.

\begin{lemma}
\label{lem:derived}
Let $T \subset G$ be a maximal torus; let $T^\der = T \cap G^\der$.  Then
\begin{equation}
\label{eqderived}
\frac{L_\ell(M_G^\vee(1))}{L_\ell(1, \rep_T)}=\frac{L_\ell(M_{G^\der}^\vee(1))}{L_\ell(1,
  \rep_{T^\der})}.
\end{equation}
\end{lemma}

\begin{proof}
The motive $M_H$ of a reductive group $H$, and thus $L_\ell(M_H\dual(1))$,
depends on $H$ only up to isogeny \cite[Lemma 2.1]{gross:motive}.
Since $G$ is isogenous to $Z \times G^\der$, $$L_\ell(M_G\dual(1)) =
L_\ell(M_Z\dual(1))L_\ell(M_{G^\der}\dual(1)).$$  Because $G^\der \cap
Z$ is finite 
\cite[(3.1)]{langlands-frenkel-ngo}, so is $T^\der \cap Z$.
Therefore, the natural map $T^\der \ra T/Z$ is an isogeny onto its
image.  For dimension reasons it is an actual isogeny, and induces an
isomorphism  $X^*(T^\der)\tensor\rat\iso X^*(T/Z)\tensor\rat$ of $\gal(\rat_\ell)$-modules.
Therefore, $L(s,\rep_{T^\der}) = L(s,\rep_{T/Z})$, and thus
\[
L(s,\rep_T) = L(s,\rep_{T/Z}) L(s,\rep_Z) =
L(s,\rep_{T^\der})L(s,\rep_Z).
\]
Identity \eqref{eqderived} is now immediate.
\end{proof}

\subsection{Weyl discriminants and measures} 

Our next immediate goal is to find an explicit constant $d(\gamma)$
such that $\mu_{\gamma}^\can = d(\gamma) \mu^{\geom}_\gamma$. We note
that a similar calculation is carried out in
\cite{langlands-frenkel-ngo}. However, the notation there is slightly
different, and the key proof in \cite{langlands-frenkel-ngo} only
appears for the field of complex numbers; hence, we decided to include
this calculation here.

Let $G$ be a split reductive group over $\rat_\ell$;
choose a split maximal torus and associated
root system $R$ and set of positive roots $R^+$. 

\begin{defn}
Let  $\gamma\in G(\Q_\ell)$, let $T$ be the centralizer of $\gamma$, and $\ft$ the Lie algebra of $T$. Then the discriminant of $\gamma$ is
$$D(\gamma)=\prod_{\alpha\in R}(1-\alpha(\gamma))=\det(I-\operatorname{Ad}(\gamma^{-1}))\vert{\fg/\ft}.$$
\end{defn}

\subsubsection{Weyl integration formula, revisited}
As pointed out in \cite{langlands-frenkel-ngo} (the paragraph above
equation (3.28)), since both $\mu^\can_\gamma$ and $\mu^\geom_\gamma$
are invariant under the center, it suffices to consider the case
$G=G^\der$.  So for the moment, let us assume that the group $G$ is
semisimple and simply connected; let $\phi\in C_c^\infty(\rat_\ell)$.

On one hand, the Weyl integration formula (we write a group-theoretic
version of the formulation for the Lie algebra in \cite[\S 7.7]{kottwitz:clay})
asserts that
\begin{equation}
\label{eqweylintegration}
\int_{G(\rat_\ell)} \phi(g) \abs{d\omega_G} =\sum_T \frac 1{\abs{W_T}}\int_{T(\rat_\ell)}\abs{D(\gamma)}
\int_{T(\rat_\ell)\backslash G(\rat_\ell)} \phi(g^{-1}\gamma g ) d\abs{\omega_{T\backslash G}}
\, d\abs{\omega_T},
\end{equation}
by our definition of the measure $d\abs{\omega_{T\backslash G}}$.
(Here, the sum ranges over a set of representatives for
$G(\rat_\ell)$-conjugacy classes of maximal $\rat_\ell$-rational tori
in $G$, and $W_T$ is the finite group $W_T = N_G(T)(\rat_\ell)/T(\rat_\ell)$.)

On the other hand we have, by definition of the geometric measure, 
$$
\int_{G(\rat_\ell)} \phi(g)\, d\abs{\omega_G}=\int_{A(\rat_\ell)}\int_{\fc^{-1}(a)}\phi(g)\,d\abs{\omega_\gamma^{\geom}(g)} 
\,\abs{d\omega_A}.$$
To compare the two measures, we need to match the integration over
$A(\rat_\ell)$ with the sum of integrals over tori. 

Up to a set of measure zero, $A(\rat_\ell)$ is a disjoint union of
images of $T(\rat_\ell)$, as $T$ ranges over the same set as in
\eqref{eqweylintegration}; and for each such $T$, the restriction of $\fc$ to $T$
is $\abs{W_T}$-to-one.

It remains to compute the Jacobian of this map for a given $T$.  Over
the algebraic closure of $\rat_\ell$ this calculation is done, for example, in
\cite{kottwitz:clay}*{\S 14}; over $\rat_\ell$, this only applies to the split
torus $T^\spl$.  The answer over the algebraic closure is
$c_T\prod_{\alpha>0}(\alpha(x)-1)$, where $c_T\in \bar F^\times$ is a
constant (which depends on the torus $T$).  We compute
$\abs{c_T}_\ell$ in the special case where $T$ comes from a
restriction of scalars in $\GL_2$.

\begin{lemma}
Let $T$ be a torus in $\GL_2(\rat_\ell)$, and let $c_T$ be the constant defined above. 
Then $\abs{c_T}_\ell =1$ if $T$ is split or splits over an unramified extension, and $\abs{c_T}=\ell^{-1/2}$ if $T$ splits over a ramified quadratic extension. 
In particular, if 
$\gamma_0\in \GL_2(\rat)$ and $T=\res_{K/\rat}\Gm$ is the centralizer of $\gamma_0$ as in \S \ref{sub:defisog}, then 
$\abs{c_T}=\abs{\Delta_K}_\ell^{-1/2}$. 
\end{lemma}

\begin{proof} We prove the lemma by direct calculation for $\GL_2$.
First, let us compute $\abs{c_{T}}$ for the split torus.
Here we can just compute the Jacobian of the map $T^\der\to T^\der/W$ by hand. Since we are working with invariant differential forms, we can just do the Jacobian calculation on the Lie algebra; it suffices to
compute the Jacobian of the map from $\ft$ to $\ft/W$.
Choose coordinates on the split torus in $\SL_2=\GL_2^\der$, so that elements of $\ft$ are diagonal matrices with entries $(t, -t)$; then the canonical measure on $\ft$ is nothing but $dt$. Now, the coordinate on $\ft/W$ is $y=-t^2$; the form $\omega_{\A^1}$ is $dx$.  
The Jacobian of the change of variables from $\ft/W$ to $\A^1$ is $-2t$.
Thus, for the split torus $c=-1$: note that $2t$ is the product of positive roots (on the Lie algebra). Thus, 
$\abs{c_T}=1$. 

Now, consider a general maximal torus $T$ in $\GL_2$. 
Let $T^\spl$ be a split maximal torus; we have shown that $\abs{c_{T^\spl}}=1$. 
The torus $T$ is conjugate to $T^\spl$ over a quadratic  field extension $L$. 
Let us briefly denote this  conjugation map by $\psi$.
Then the map $\fc\vert T$ can thought of as the conjugation $\psi:T \to T^\spl$ (defined over $L$) followed by 
 the map $\fc\vert{T^\spl}$. 
Then 
$$c_T= c_{T^\spl}\frac{\omega_T}{\psi^\ast(\omega_{T^\spl})},$$
where $\psi^\ast(\omega_{T^\spl})$ is the pullback of the canonical volume form on $T^\spl$ under $\psi$ and the ratio $\frac{\omega_T}{\psi^\ast(\omega_{T^\spl})}$ is a constant in $L$.  We thus have 
\begin{equation}\label{eq:psi}
c_T=\abs{\frac{\omega_T}{\psi^\ast(\omega_{T^\spl})}}_L,
\end{equation}
where $\abs{\cdot}_L$ is the unique extension of the absolute value on $\Q_\ell$ to $L$. 

At this point this is just a question about two tori, no longer requiring Steinberg section, and so we pass back to working with the group $\GL_2$ rather than $\SL_2$. 
Now $T$ is obtained by restriction of scalars from $\Gm$, and so
we can compute $\psi^\ast(\omega_{T^\spl})$ by hand.
By definition, $T=\res_{L/\Q_\ell}\G_m$; $T^\spl =\Gm\times \Gm$. The form $\omega_{T^\spl}$ is 
$$\omega_{T^\spl}=\frac{du}{u}\wedge\frac{dv}{v},$$ 
where we denote the coordinates on $\Gm\times\Gm$ by $(u,v)$. 
Let $L=\Q_\ell(\sqrt{\epsilon})$, where $\epsilon$ is a non-square in $\Q_\ell$ 
(assume for the moment that $\ell\neq 2$). 
Then every element of $T$ is conjugate in $\GL_2(\Q_\ell)$ to 
$\left[\begin{smallmatrix} x & \epsilon y\\ y & x\end{smallmatrix}\right]$, and using $(x, y)$ as the coordinates on 
$T$, the map  $\psi$ can b written as $\psi(x, y)=(x+\sqrt{\epsilon}y, x-\sqrt{\epsilon}y)$. 
Then one can simply compute 
$$\psi^\ast(\frac{du}u\wedge\frac{dv}{v})= 2\sqrt{\epsilon}\frac{dx \wedge dy}{x^2-\epsilon y^2} =2\sqrt{\epsilon}\omega_T.$$
Thus we get (for $\ell\neq 2$), 
\begin{equation*}
\abs{c_T}_\ell =\abs{2\sqrt{\epsilon}}_L =\begin{cases} 1 &\quad  L \text{ is unramified}\\ 
                                                  \sqrt{\ell}& \quad L \text{ is ramified},
                                         \end{cases}
\end{equation*}
which completes the proof of the lemma in the case $\ell\neq 2$. 

There is, however, a better argument, which also covers the case $\ell=2$. 
Namely, to find the ratio $\abs{\frac{\omega_T}{\psi^\ast(\omega_{T^\spl})}}_L$ of (\ref{eq:psi}),
we just need to find the ratio of the volume of  
$\cT^\circ(\Z_\ell)$ with respect to the measure $\abs{d\omega_T}$ to its volume with respect to 
$\abs{d\psi^\ast(\omega_{T^\spl})}$.
This is, in fact, the same calculation as the one carried out 
in \cite{weil:adeles}*{p.22 (before Theorem 2.3.2)}, and the answer is that the convergence factors for the pull-back of the form $\omega_{T^\spl}$ to the restriction of scalars is 
$(\sqrt{\abs{\Delta_K}_\ell})^{\dim(\Gm)}$, in this case.
\end{proof}

Finally, summarizing the above discussion, we obtain 
\begin{proposition}\label{prop:comp}
Let $\gamma\in \gl_{2}(\Q)$ be a regular element. Let $T$ be the centralizer of $\gamma$, and let $K$ be as in 
 \S\ref{sub:defisog}.
 Abusing notation, we also denote by $\gamma$ the image of $\gamma$ in $\GL_2(\Q_\ell)$ for every finite prime $\ell$.
Then for every finite prime $\ell$, 
\begin{equation*}
\mu_\gamma^{\geom}=  \frac{L_\ell(1, \rep_T)}{L_\ell(M_G^\vee(1))} \abs{\Delta_K}_\ell^{-1/2} \abs{D(\gamma)}_\ell^{1/2}\mu_\gamma^\can
\end{equation*}
as measures on the orbit of $\gamma$. 
\end{proposition}

\section{The global calculation}
\label{sec:global}
In this section, we put all the above local comparisons together, and thus show that Gekeler's formula reduces to a special  case of the formula of Langlands and Kottwitz. In the process we will need a formula for the global volume term that arises in that formula. 
We are now in a position to give a new proof of Gekeler's theorem, and of its generalization to arbitrary finite fields.

\begin{theorem}\label{th:main}
Let $q$ be a prime power, and let $a$ be an integer with $\abs a \le 2 \sqrt q$ and $\gcd(a,p) =1$.  The number of elliptic curves over $\ff_q$ with trace of Frobenius $a$ is
\begin{equation}
\label{eqmain}
\wnum I(a, q)=\frac{\sqrt{q}}2\nu_{\infty}(a, q)\prod_\ell \nu_\ell(a, q).
\end{equation}
\end{theorem}

Here, $\nu_\ell(a,q)$ (for $\ell\not =p$), $\nu_p(a,q)$, and $\nu_\infty(a,q)$ are defined, respectively, in \eqref{eqnewgekeler}, \eqref{eqnewgekeler-p}, and \eqref{eqgekeler-infinity}, and the weighted count $\wnum I(a,q)$ is defined in \eqref{eqdefwnum}.

\begin{proof}
Recall the notation surrounding $\gamma_0$ and $\delta_0$ established in \S \ref{sub:defisog}.
Given Proposition \ref{proplk}, it suffices to show that the right-hand side of \eqref{eqmain} calculates the right-hand side of \eqref{eqlk}.

Let $G=\GL_2$.
First, let
\[
\phi^p = \tensor_{\ell\not = p} \one_{G(\integ_\ell)}
\]
be the characteristic function of $G(\hat\integ^p_f)$ in $G(\aff^p_f)$.  The first integral appearing in \eqref{eqlk} is equal to 
$$O_{\gamma_0}(\phi^p)=\int_{G(\A^{p})}\phi^p \abs{d\omega_G}= \prod_{\ell\neq p} O^\can(\one_{G(\integ_\ell)}).$$
 
Combining Corollary \ref{cor:gekeler-steinberg}, relation \eqref{eqderived} 
and  Proposition \ref{prop:comp}, we get, for $\ell\neq p$,  
$$
\begin{aligned}
\nu_{\ell}(a, q) & =\frac{\ell^{3}}{\#G^\der(\F_\ell)} O^\geom_{\gamma_0}(\one_{G(\Z_\ell)})
= \frac{\ell^{3}}{\#G^\der(\F_\ell)}\frac{L_\ell(1, \rep_{T^\der})}{L_\ell(M_{G^\der}^\vee(1))} 
\abs{\Delta_K}_\ell^{-1/2} \abs{D(\gamma_0)}_\ell^{1/2}
O^{\can}_{\gamma_0}(\one_{G(\Z_\ell)})\\
&= L_\ell(1, \rep_{T^\der})\abs{D(\gamma_0)}_\ell^{1/2} \abs{\Delta_K}_\ell^{-1/2} O_{\gamma_0}^\can(\one_{G(\Z_\ell)}).
\end{aligned}
$$

Second, let $\phi_q$ be the characteristic function of $G(\integ_p) \begin{pmatrix}1&0\\0&q\end{pmatrix} G(\integ_p)$ in $G(\rat_p)$, and let $\phi_{p,q}$ be the characteristic function of $G(\integ_q) \begin{pmatrix}1&0\\0&p\end{pmatrix} G(\integ_q)$ in $G(\rat_q)$.  Using Lemmas \ref{lem:l=p,g=1} and \ref{lem:basechange}, we find that 
\begin{align*}
\nu_p(a,q) &= \frac{p^3}{\#G^\der(\ff_p)}O^\geom_{\gamma_0}(\phi_q) \\
&= \frac{p^3}{\#G^\der(\ff_p)} \frac{L_p(1,\rep_{T^\der})}{L_p(M_{G^\der}\dual(1))} \abs{\Delta_K}_p^{-1/2} \abs{D(\gamma_0)}_p^{1/2} O^\can_{\gamma_0}(\phi_q) \\
&= \frac{p^3}{\#G^\der(\ff_p)} \frac{L_p(1,\rep_{T^\der})}{L_p(M_{G^\der}\dual(1))} \abs{\Delta_K}_p^{-1/2} \abs{D(\gamma_0)}_p^{1/2} TO_{\delta_0}^\can(\phi_{p,q}).
\end{align*}

Taking a product over all finite primes, we obtain:
\begin{equation}
\label{eqprodfinite}
\prod_{\ell < \infty} \nu_\ell(a,q) = L(1,\rep_{T^\der}) \sqrt{\frac{\abs{\Delta_K}}{\abs{D(\gamma_0)}}}TO^\can_{\delta_0\sigma}(\phi_{p,q}) O^\can_{\gamma_0}(\phi^p).
\end{equation}
Recall that $f_0(T)$, the characteristic polynomial of $\gamma_0$, is $f_0(T) = T^2-aT+q$.  The (polynomial) discriminant of $f_0(T)$ and the  (Weyl) discriminant of $\gamma_0$ are related by
$\abs{D(\gamma_0)\det(\gamma_0)} = \abs{\operatorname{disc}(f_0)} = 4q-a^2$.  Consequently,
\[
\sqrt q \nu_\infty(a,q) = \oneover\pi\sqrt{\abs{D(\gamma_0)}}.
\]
Since $L(1,\rep_{T^\der}) = L(1,\rep_{T/Z})$ (Lemma
\ref{lem:derived}),  \eqref{eqmain} follows from \eqref{eqprodfinite}
and Proposition \ref{lem:arch} below, which is proved in the Appendix. 
\end{proof}

\begin{proposition}
\label{lem:arch}
We have
\begin{equation}\label{eq:desired}
\frac{\sqrt{\abs{\Delta_K}}}{2\pi} L(1, \rep_{T/Z}) = 
\vol( T(\Q)\backslash T(\A_f)).
\end{equation}
\end{proposition}

\begin{proof} First, we note that $L(s, \rep_{T/Z})$ coincides with $L(s, K/\Q)$, the Dirichlet $L$-function attached to the quadratic character of $K$.
Now the proposition is obtained by combining Lemma \ref{applem5} with the analytic class number formula \eqref{appeq5}.
\end{proof}

\newpage
\appendix

\section{Orbital integrals and measure conversions}
\label{appali}
\begin{center}
S.~Ali Altu\u{g} 
\end{center}
In this appendix we relate certain calculations of \cite{Altug:2015aa}, initially intended for a different setting, to Gekeler's product formula \eqref{eqgekeler}. The strategy for the proof will follow the same lines as in the main text; the major differences are at the calculations of measure conversion factors and volumes. Since this appendix is to be complementary to the main text we will not aim for generality and simply take $q=p$, which is the case considered in \cite{gekeler03}. More precisely, we will be proving:

\begin{thm*}
{\textit{(Theorem \ref{th:main}, $q=p$-prime case)}} Let $p$ be a
prime and let $\mathbb{F}_p$ denote the finite field  with
$p$ elements. Let $a\in\mathbb{Z}$ such that $\abs{a}\leq 2\sqrt{p}$,
$p\nmid a$ and let $I(a,p)$ denote the isogeny class of elliptic curves over $\mathbb{F}_p$ with trace of their Frobenius equals $a$. For each finite prime $\ell$, let $v_\ell(a,p)$ be the local probabilities defined by $(1.1)$ and $(1.2)$. Then
\[\wnum I(a,p)=\frac{\sqrt{4p-a^2}}{2\pi}\prod_{\ell\neq\infty}v_\ell(a,p).\]
\end{thm*}

The proof consists of four steps. The first two steps, relating $\wnum I(a,p)$ and the local densities to orbital integrals (with respect to different measures (!)), are the same as in the text, so instead of giving a detailed exposition of these we simply refer to the relevant parts of the article. The third step is to find the constant of proportionality in the two measure normalizations and to calculate the global volume factor that appears in the Langlands-Kottwitz formula, and the final step is to put everything together via the class number formula.

\subsection{Steps 1 \& 2: Sizes of isogeny classes, orbital integrals, and local densities }\label{appsec0.1}

\subsubsection{Notation}
  Let $G:=\gl_2$ and let $\mathbb{A}$ be the adeles ${\mathbb
    A}=\mathbb{A}_{\mathbb{Q}}$. Following the notation of \S3.2 and
  3.3, for any prime $\ell\neq p$, let $\phi_{0}$ denote the
  characteristic function of the $G(\mathbb{Z}_\ell)$. For $\ell=p$,
  let $\phi_p$ be the characteristic function of
  $G(\mathbb{Z}_p)\left(\begin{smallmatrix}p&\\
      &1\end{smallmatrix}\right)G(\mathbb{Z}_p)$. Let $\gamma\in
  G(\mathbb{Q})$ be such that $\tr(\gamma)=a$ and $\det(\gamma)=p$. Note
  that since $a\in\mathbb{Z}$ and $\abs{a}\leq 2\sqrt{p}$, $\gamma$ is
  regular semisimple and its centralizer, $G_{\gamma}$, is a
  torus. Let us denote this torus by $T$. Finally, for any finite
  prime $\ell$ and $\phi\in C_c^{\infty}(G(\mathbb{Q}_\ell))$ let the
  orbital integrals $O_{\gamma}^\geom(\phi)$ and
  $O_{\gamma}^{\can}(\phi)$ be defined as in definitions 3.1 and 3.2
  respectively. We remark that the difference between
  $O_{\gamma}^\geom(\phi)$ and $O_{\gamma}^{\can}(\phi)$ is in the
  chosen measure on the orbit of $\gamma$. Instead of going through
  the details of these measures we simply refer to Sections \ref{sub:steinberg} and \ref{sub:measures} of
  the paper as well as \S3.3 of \cite{langlands-frenkel-ngo}.

\subsubsection{Sizes of isogeny classes to orbital integrals:
    Langlands-Kottwitz formula.}
For each finite prime $\ell$ let $d\mu_{G,\ell}^{\can}$ denote the measure on $G(\mathbb{Q}_\ell)$ normalized to give volume $1$ to $G(\mathbb{Z}_\ell)$. On the centralizer $T(\mathbb{Q}_\ell):=G_{\gamma}(\mathbb{Q}_\ell)$ choose any measure $d\mu_{T,\ell}$. Let $d\bar{\mu}_\ell$ denote the quotient measure $d\mu_{G,\ell}^{\can}/d\mu_{T,\ell}$, and define $d\mu_{T,f}=\otimes_{\ell\neq\infty}d\mu_{T,\ell}$. 

Proposition 2.1 combined with Lemma 3.7 states that
\begin{equation}\label{appeq1}
\#I(a,p)=\vol(\gamma,d\mu_{T,f})\int_{T(\mathbb{Q}_p)\backslash G(\mathbb{Q}_p)}\phi_p(g^{-1}\gamma g)d\bar{\mu}_p(g)\prod_{\substack{\ell\text{ finite}\\ \ell\neq p}}\int_{T(\mathbb{Q}_\ell)\backslash G(\mathbb{Q}_\ell)}\phi_0(g^{-1}\gamma g)d\bar{\mu}_\ell(g),\end{equation}
where 
\begin{equation}\label{appeq2}
\vol(\gamma,d\mu_{T,f}):=\int_{T(\mathbb{Q})\backslash T(\mathbb{A}_f)}d\mu_{T,f},
\end{equation} 
and $\mathbb{A}_f=\prod^{'}_{\ell\text{ finite}}\mathbb{Q}_\ell$ denotes
the finite adeles. We also remark that the orbital integrals and the
volume in \eqref{appeq1} depend only on the conjugacy class of
$\gamma$. Since, for semisimple elements, conjugacy in $ G(\mathbb{Q})$ is equivalent to
having the same characteristic polynomial (note that this uses the
fact that $G=\gl_2$), \eqref{appeq1} is well defined.

\subsubsection{Orbital integrals to local densities.}
  As in \S\ref{appsec0.1} fix a prime $p$, an integer $\abs{a}\leq
  2\sqrt{p}$, and $\gamma\in G(\mathbb{Q})$ such that
  $\tr(\gamma)=a$ and $\det(\gamma)=p$. Let the local densities
  $v_\ell(a,p)$ be defined by (1.1) and (1.2), and for any $\phi\in
  C_c^{\infty}(\mathbb{Q}_\ell)$ let $O^\geom_{\gamma}(\phi)$ be the
  orbital integrals, with respect to the geometric measure (cf. \S2.2
  and 3.1), defined as in definition 3.1. Then, recalling that $\#
  SL_2(\mathbb{F}_\ell)=\ell^3\zeta_\ell(2)^{-1}$, Corollary 3.5 and
  Lemma 3.6 can be stated as
\begin{align}
v_\ell(a,p)&=\zeta_\ell(2)O^\geom_{\gamma}(\phi_0)\label{appeq3}\\
v_p(a,p)&=\zeta_p(2)O^\geom_{\gamma}(\phi_p)\label{appeq4},
\end{align}
where $\phi_0$ and $\phi_1$ are as in \S\ref{appsec0.1}, and $\zeta_\ell(s)=1/(1-\ell^{-s})$.

\subsection{Step 3: Measure conversions and orbital integrals}

Since the orbital integrals $O^\geom_{\gamma}(\phi)$ and
$O^{\can}_{\gamma}(\phi)$ are defined with respect to different measure
normalizations on the orbit of $\gamma$, in order to compare
\eqref{appeq3} and \eqref{appeq4} to \eqref{appeq1} we will need to
relate integration against the quotient measure
$d\bar{\mu}_\ell=d\mu^{\can}_{G,\ell}/d\mu_{T,\ell}$ to integration
against $d\mu^\geom_\ell$. In order to do so, we start with a lemma
that provides the conversion factor for a general quotient measure,
$d\mu_{T\backslash G,\ell}:=d\mu_{G,\ell}/d\mu_{T,\ell}$.

\begin{lemma}\label{applemma0}Let $\ell$ be a finite prime, $\gamma\in G(\mathbb{Q})$, and $T=G_{\gamma}$ be as above.  For any Haar measures $d\mu_{G,\ell}$ on $G$ and $d\mu_{T,\ell}$ on $T$ let $\vol(d\mu_{G,\ell})$ be the volume of $G(\mathbb{Z}_\ell)$ with respect to $d\mu_{G,\ell}$ and similarly let $vol(d\mu_{T,\ell})$ denote the volume of $T(\mathbb{Z}_\ell)$. Let $d\mu_{T\backslash G,\ell}$ denote the quotient measure $d\mu_{G,\ell}/d\mu_{T,\ell}$. Let $\omega_G$ and $\omega_T$ be non-vanishing algebraic top degree forms on $G$ and $T$ respectively. Denote the measured associated to $\omega_{G}$ and $\omega_T$ by $\abs{\omega_{G}}_\ell$ and $\abs{\omega_T}_\ell$ respectively, and denote the induced quotient measure on $T\backslash G$ by $\abs{\omega_{T\backslash G}}_\ell$. Finally, let $d\mu_{\gamma,\ell}^\geom$ be the geometric measure of \S3.1.3 induced by $\omega_{G}$. Then,
\[d\mu_{\gamma,\ell}^\geom=\sqrt{\abs{D(\gamma)}_\ell}\frac{\vol(\abs{\omega_G}_\ell)\vol(d\mu_{T,\ell})}{\vol(\abs{\omega_T}_\ell)\vol(d\mu_{G,\ell})}d\bar{\mu}_{T\backslash G,\ell},\]
where, by abuse of notation, $\abs{D(\gamma)}=\abs{\tr(\gamma)^2-4\det(\gamma)}$.

\end{lemma}

\begin{proof}
By equation (3.30) of \cite{langlands-frenkel-ngo}, we have
\[d\mu_{\gamma,\ell}^\geom=\sqrt{\abs{D(\gamma)}_\ell}\abs{\omega_{T\backslash G}}_\ell,\]

where we note that the left hand side of (3.30) of loc. cit. is what
we denoted by $d\mu_{\gamma}^\geom$. Here, we need to remind the
reader that both measures are invariant under multiplication by
central elements so without loss of generality we can assume that
$G=G^\der$ (i.e. $D(\gamma)=D_{G^\der}(\gamma)\det(\gamma)$).   On the
other hand, since the Haar measure is unique up to a constant we have
$\abs{\omega_G}_\ell=c_\ell(G)d\mu_{G,\ell}$ and
$\abs{\omega_T}_\ell=c_\ell(T)d\mu_{T,\ell}$. The constants can be
calculated easily by volumes of the integral points:
\[c_\ell(G)=\frac{\vol(\abs{\omega_G}_\ell)}{\vol(d\mu_{G,\ell})}\qquad\text{
  and }\qquad c_\ell(T)=\frac{\vol(\abs{\omega_T}_\ell)}{\vol(d\mu_{T,\ell})}.\]
Therefore, the quotient measures $d\bar{\mu}_{T\backslash G,\ell}$ and $\abs{\omega_{T\backslash G}}_\ell$ are related by
\[\abs{\omega_{T\backslash G}}_\ell=\frac{c_\ell(G)}{c_\ell(T)}d\bar{\mu}_{T\backslash G,\ell}.\]
The lemma follows.
\end{proof}

As an immediate corollary to Lemma \ref{applemma0} we get

\begin{corollary}\label{appcor1} Let $d\mu_{G,\ell}^{\can}$ and $d\mu_{T,\ell}^{\can}$ be normalized to give measure $1$ to $G(\mathbb{Z}_\ell)$ and $T(\mathbb{Z}_\ell)$ respectively, and let the rest of the notation be as in Lemma \ref{applemma0}. Then
\[d\mu_{\gamma,\ell}^\geom=\sqrt{\abs{D(\gamma)}_\ell}\frac{\vol(\abs{\omega_G}_\ell)}{\vol(\abs{\omega_T}_\ell)}d\bar{\mu}_{T\backslash G,\ell}.\]

\end{corollary}

We now quote a result of \cite{Langlands:2013aa}, where $\vol(\abs{\omega_G}_\ell)$ and $\vol(\abs{\omega_T}_\ell)$ are calculated for a specific pair of forms $\omega_{G}$ and $\omega_{T}$.  Let
\begin{equation}
\label{eqdefomega}
\omega_G=\frac{ d\alpha d\beta d\gamma d\delta}{\alpha\delta-\beta\gamma}\qquad\text{ and }\qquad\omega_T=\frac{d\gamma_1 d\gamma_2}{\gamma_1\gamma_2}
\end{equation}
where we have chosen coordinates
\begin{align*}
G &= \left\{\begin{pmatrix}\alpha&\beta\\\gamma&\delta\end{pmatrix}: \alpha\delta-\beta\gamma\not = 0 \right\}\\
T &= \left\{\begin{pmatrix}\gamma_1&\\&\gamma_2\end{pmatrix}: \gamma_1\gamma_2\not=0\right\}.
\end{align*}
Note that $\omega_G$ is the same form studied in Section \ref{sub:measures}.

\begin{lemma}\label{applemma2} 
For $\omega_G$ and $\omega_T$ as in \eqref{eqdefomega},
\begin{align*}
\vol(\abs{\omega_G}_\ell)&=\zeta_\ell(1)^{-1}\zeta_\ell^{-1}(2)\\
\vol(\abs{\omega_T}_\ell)&=\sqrt{\abs{\Delta_K}_\ell}\begin{cases}\zeta_\ell(1)^{-2}& \text{$K/\mathbb{Q}$ is split at $\ell$}\\ \zeta_\ell(2)^{-1}&  \text{$K/\mathbb{Q}$ is unramified at $\ell$}\\ \zeta_\ell(1)^{-1}& \text{$K/\mathbb{Q}$ is ramified at $\ell$}\end{cases},
\end{align*}
where $K/\mathbb{Q}$ is the quadratic extension where $T$ splits over and $\Delta_K$ is the discriminant of $K$ (so in the split and unramified cases $\abs{\Delta_K}_\ell$ is $1$).
\end{lemma}

\begin{proof}The result for odd primes $\ell$ is given on pages 41 and
  42 of \cite{Langlands:2013aa}. The case for $\ell=2$ follows the
  same lines.   The only point to keep in mind is the extra factor of $2$ that appears in the calculation of the differential form on page 42 of \cite{Langlands:2013aa}; we leave the details to the reader.

\end{proof}

\begin{corollary}\label{appcor2} Let the measure choices be as in Lemma \ref{applemma2}. Then for any $\phi\in C_c^{\infty}(G(\mathbb{Q}_\ell))$ and any regular semi-simple $\gamma$,

\[O^{\can}_{\gamma}(\phi)=\frac{\sqrt{\abs{\Delta_K}_\ell}\zeta_\ell(2)}{\sqrt{\abs{D(\gamma)}_\ell}L_\ell(1,K/\mathbb{Q})}O^\geom_{\gamma}(\phi).\]
\end{corollary}

\begin{proof} By Lemma \ref{applemma2} we have 

\[\frac{\vol(\abs{\omega_G}_\ell)}{\vol(\abs{\omega_T}_\ell)}=\tfrac{1}{\sqrt{\abs{\Delta_K}_\ell}}\begin{cases}\frac{1}{(1+1/q_\ell)}&split\\ \frac{1}{(1-1/q_\ell)}& unramified \\ (1-1/q_\ell^2)& ramified\end{cases}=\frac{L_\ell(1,K/\mathbb{Q})}{\sqrt{\abs{\Delta_K}_\ell}\zeta_\ell(2)}.\]
The corollary then follows from Corollary \ref{appcor1}.
\end{proof}

The last step is to calculate the global volume term.

\begin{lemma}\label{applem5}
Let $(a,p)$ be such that $a^2-4p<0$. Let $d\mu_{T,\ell}^{\can}$ be the Haar measure normalized to give measure $1$ to $T(\mathbb{Z}_\ell)$ and set $d\mu_{T,f}^{\can}:=\otimes_{\ell\neq\infty}d\mu_{T,\ell}^{\can}$. Then,
\[\vol({\gamma},d\mu_{T,f}^{\can})=\frac{h_K}{w_K},\]
here $K/\mathbb{Q}$ is the quadratic extension where $T$ splits, $w_K$ is the number of roots of unity in $K$, and $h_K$ is its class number.

\end{lemma}

\begin{proof}By identifying $T=G_{\gamma}$ with $\mathbb{G}_{m}$ over the quadratic extension $K$ we have 
\[\vol({\gamma},d\mu_{T,f}^\can)=\mu_{T,f}^\can(T(\mathbb{Q})\backslash T(\mathbb{A}_f))=\mu_{K,f}^\can(K^{\times}\backslash \mathbb{A}_{K,f}^{\times}),\]
where the measure on the right is such that $\mu^\can_{K,f}(O_{v}^{\times})=1$ for each place $v$. Let $\hat{\mathcal{O}}_K^{\times}=\prod_{v}\mathcal{O}_v^{\times}$. Recall that
\[1\rightarrow (K^{\times}\cap \hat{\mathcal{O}}^{\times}_K)\backslash \hat{\mathcal{O}}^{\times}_K\rightarrow K^{\times}\backslash \mathbb{A}_{K,f}^{\times}\rightarrow \operatorname{Cl}(K)\rightarrow 1,\]
which implies that $\mu(K^{\times}\backslash \mathbb{A}_{K,f}^{\times})=h_K\mu((K^{\times}\cap \hat{\mathcal{O}}^{\times}_K)\backslash \hat{\mathcal{O}}^{\times}_{K})=\frac{h_K}{w_K}$.

\end{proof}

\subsection{Proof of Theorem \ref{th:main}}

Let us begin by recalling Dirichlet's class number formula for an imaginary quadratic field $K/\mathbb{Q}$:
\begin{equation}\label{appeq5}
L\left(1,K/\mathbb{Q}\right)=\frac{2 \pi h_K}{w_K\sqrt{\abs{\Delta_K}}},
\end{equation}
where $\Delta_K$ the discriminant of $K$, $w_K$ is the number of roots of unity in $K$, and $h_K$ is its class number.

At this point the result essentially follows from a combination of Corollary \ref{appcor2}, Lemma \ref{applem5}, and \eqref{appeq5}. We give the details:

By \eqref{appeq1} we have
\begin{equation*}
\wnum I(a,p)=\vol_f(\gamma)O^\can_{\gamma}(\phi_p)\prod_{\substack{ l\neq p}}O_{\gamma}^\can(\phi_0).
\end{equation*}
Substituting Corollary \ref{appcor2} and Lemma \ref{applem5} into the above equation (and noting that the product converges only conditionally so should be calculated in the given order) gives
\[\#I(a,p)=\frac{h_K}{w_K}\frac{\sqrt{\abs{\Delta_K}_p}\zeta_p(2)}{\sqrt{\abs{D(\gamma)}_p}L_p(1,K/\mathbb{Q})}O^\geom_{\gamma}(\phi_p)\prod_{\substack{\ell\neq p}}\frac{\sqrt{\abs{\Delta_K}_\ell}\zeta_\ell(2)}{\sqrt{\abs{D(\gamma)}_\ell}L_\ell(1,K/\mathbb{Q})}O^\geom_{\gamma}(\phi_0).\]

Then, substituting \eqref{appeq3} and \eqref{appeq4}, we get
\begin{align*}
\wnum I(a,p)&=\frac{h_K}{w_K}\prod_{\substack{\ell\text{ finite}}}\frac{\sqrt{\abs{\Delta_K}_\ell}}{\sqrt{\abs{D(\gamma)}_\ell}L_\ell(1,K/\mathbb{Q})}v_\ell(a,p)\\
&=\frac{h_K}{w_K}\frac{1}{L(1,K/\mathbb{Q})}\prod_{\ell\text{ finite}}\frac{\sqrt{\abs{\Delta_K}_\ell}}{\sqrt{\abs{D(\gamma)}_\ell}}v_\ell(a,p).
\end{align*}
Finally, using \eqref{appeq5} and the product formula (i.e. $\prod_\ell \abs{\Delta}_\ell=\prod_\ell\abs{D(\gamma)}_\ell=1$, where the product is over \emph{all} primes $\ell$) in the above equation we get,
\begin{align*}
\#I(a,p)&=\frac{\sqrt{\abs{\Delta_K}}}{2\pi}\frac{\sqrt{\abs{D(\gamma)}}}{\sqrt{\abs{\Delta_K}}}\prod_{\ell\text{ finite}}v_\ell(a,p)\\
&=\frac{\sqrt{4p-a^2}}{2\pi}\prod_{\ell\text{ finite}}v_\ell(a,p),
\end{align*}
where the absolute values $\abs{\cdot}$ are the archimedean absolute values and we used the fact that $\abs{D(\gamma)}=4p-a^2$.

\hfill{$\square$}

{\sc Columbia University, New York, NY}

{\it E-mail address:} {\tt altug@math.columbia.edu}

\newpage
\bibliographystyle{amsalpha}
\bibliography{biblio}

\end{document}